\documentclass[11pt]{amsart}
\usepackage{enumerate}
\usepackage{amssymb}
\usepackage{amsthm}
\usepackage{amsmath}
\usepackage{graphicx}
\usepackage{latexsym}
\usepackage{fontenc}
\usepackage{color}
\usepackage{url}
\usepackage{hyperref}
\usepackage{cleveref}
\usepackage{fullpage}

\newtheorem{theorem}{Theorem}

\theoremstyle{definition}
\newtheorem{definition}[theorem]{Definition}

\newtheorem{example}[theorem]{Example}

\theoremstyle{remark}

\begin{document}

\title[A rigid origami elliptic-hyperbolic vertex duality]{A rigid origami elliptic-hyperbolic vertex duality}

\author{Thomas C. Hull}

\address{Mathematics Department, Franklin \& Marshall College, Lancaster, PA, 17604-3003  USA}

\email{thomas.hull@fandm.edu}



\maketitle

\begin{abstract}
The field of rigid origami concerns the folding of stiff, inelastic plates of material along crease lines that act like hinges and form a straight-line planar graph, called the \textit{crease pattern} of the origami. Crease pattern vertices in the interior of the folded material and that are adjacent to four crease lines, i.e. \textit{degree-4} vertices, have a single degree of freedom and can be chained together to make flexible polyhedral surfaces. Degree-4 vertices that can fold to a completely flat state have folding kinematics that are very well-understood, and thus they have been used in many engineering and physics applications. However, degree-4 vertices that are not flat-foldable or not folded from flat paper so that the vertex forms either an \textit{elliptic} or \textit{hyperbolic} cone, have folding angles at the creases that follow more complicated kinematic equations. In this work we present a new duality between general degree-4 rigid origami vertices, where dual vertices come in elliptic-hyperbolic pairs that have essentially equivalent kinematics. This reveals a mathematical structure in the space of degree-4 rigid origami vertices that can be leveraged in applications, for example in the construction of flexible 3D structures that possess metamaterial properties.
\end{abstract}

\section{Introduction}\label{sec1}

The folding of stiff, flat sheets along creases that act as hinges, commonly known as \textit{rigid origami}, has had a tremendous influence on mechanical, robotic, and metamaterial design for over a decade \cite{YouSci14,Silverberg1,Tachi15} and shows no sign of slowing down \cite{Paulino2022,Wang2023,XingYou,Song24,Zhu24}. Rigid origami makes for an attractive medium in such applications for several reasons.  First, when the unfolded material forms a flat, developable sheet, also known as \textit{Euclidean} material, then fabrication can be easier. Second, if the material can be rigidly folded continuously to a state where all the creases have zero dihedral angles, also known as the \textit{flat-foldable} case, the folded form can be stored compactly. Third, the kinematics of rigid origami can lead to novel mechanics.  See \cite{YOrigami} for examples of all three.

Mathematical advances in analytical models of rigid origami \cite{belcastro} and flexible polyhedral surfaces \cite{Bobenko} helped facilitate this surge in rigid origami applications. This in turn led to a complete understanding of the kinematics of the most simple non-trivial case: developable, flat-foldable rigid origami crease patterns that consist of a single degree-4 vertex which are then easily applied to larger crease patterns made of multiple such vertices \cite{Tachi3,Lang1}. In contrast, the algebraic development of the kinematics of general, not necessarily developable nor flat-foldable degree-4 vertices has been slower, with full equations appearing recently \cite{FHK}.

In this report we prove a new structural relationship between degree-4 rigid origami vertex foldings. Specifically, we describe a natural \textit{duality} between non-Euclidean \textit{elliptic} and \textit{hyperbolic} degree-4 vertices, which are defined as having the the sum of their sector angles between creases be $<2\pi$ or $>2\pi$, respectively.  After setting up notation and describing background results in Section~\ref{sec1.5}, we present our main result in Section~\ref{sec2}, that dual vertices have \textit{essentially equivalent kinematics}, meaning that their configuration spaces are identical up to a change in sign of the folding angles at the creases. This duality theory shows that Euclidean, flat-foldable degree-4 vertices, which are \textit{self-dual} under this scheme, play a central role in the kinematics of general degree-4 vertices, including the non-Euclidean cases. In Section~\ref{sec3} we describe some geometric interpretations of this duality concept. In Section~\ref{sec4} we show how combined dual elliptic-hyperbolic degree-4 vertices can tessellated in three-dimensional space to form flexible CW complexes, providing examples that have properties that make them good candidates for metamaterial applications. We conclude with questions for further study.

\section{Background}\label{sec1.5}

We first describe our notation. A \textit{generalized rigid origami} is a  polyhedral surface $P$ with boundary that can flex along its edges; we call the edge and vertex  set $C$ of $P$ the \textit{crease pattern} of $P$. Each edge (i.e., crease) $c_i$ in $C$ will be folded by some \textit{folding angle} $\rho_i\in [-\pi,\pi]$, where $\rho_i=0$ means the crease $c_i$ is unfolded and $\rho_i>0$ (resp. $\rho_i<0$) means $c_i$ is a \textit{valley} crease (resp. \textit{mountain} crease). If $C$ has $n$ creases then their set of viable folding angles give us points $\vec{\rho}=(\rho_1,\ldots, \rho_n)$ in $\mathbb{R}^n$ that form the \textit{configuration space} $S(C)$ of $C$. 

A crease pattern $C$ that consists of a single vertex in the interior of $P$ is completely determined by the sector angles $\alpha_i$ in order around the vertex.  Thus we may specify a degree-4 rigid origami vertex by $C=(\alpha_1,\alpha_2,\alpha_3,\alpha_4)$.  If $\sum\alpha_i=2\pi$ then we say $C$ is \textit{Euclidean}. The \textit{non-Euclidean} cases are $\sum\alpha_i<2\pi$, an \textit{elliptic} vertex, and $\sum\alpha_i>2\pi$, a \textit{hyperbolic} vertex.

We say that a single-vertex crease pattern is \textit{flat-foldable} if at least one point of the form $(\pm \pi, \pm \pi, \ldots, \pm\pi)$ is in its configuration space (i.e., the rigid origami vertex can be flexed so that all of the folding angles are $\pm\pi$, which means the folded form lies flat in a plane). A basic result of flat-foldable origami is \textit{Kawasaki's Theorem}, which states that a single-vertex crease pattern $C$ on $P$ is flat-foldable if and only if the alternating sum of the sector angles between the creases around the vertex on $P$, in order, equals zero. (See \cite{Origametry} for a proof.) We will denote these sector angles by $\alpha_i$.  

In the degree-4 case, Kawasaki's Theorem implies that $\alpha_1-\alpha_2+\alpha_3-\alpha_4=0$. If the vertex is also Euclidean, then we have $\alpha_1=\pi-\alpha_3$ and $\alpha_2=\pi-\alpha_4$.

Flat-foldable origami is useful in engineering and architecture applications because it leads to compact storage of materials. The mechanics of origami crease patterns that can flex along their creases to smoothly fold and unfold have also attracted attention in applications, especially origami made of degree-4 vertices since as a movable mechanism they possess only one degree of freedom. Controlling such rigid origami mechanisms requires knowing how all the folding angles of the creases relate to each other. The equations relating such parameters for rigid body motions are referred to as \textit{kinematic equations} in mechanical engineering, and they can be notoriously complicated and difficult capture efficiently; see \cite{Uchida} for example. It is surprising, then, that the equations governing the folding angles of Euclidean, flat-foldable degree-4 vertices are so elegant:

\begin{figure}
\centerline{
\includegraphics[width=\linewidth]{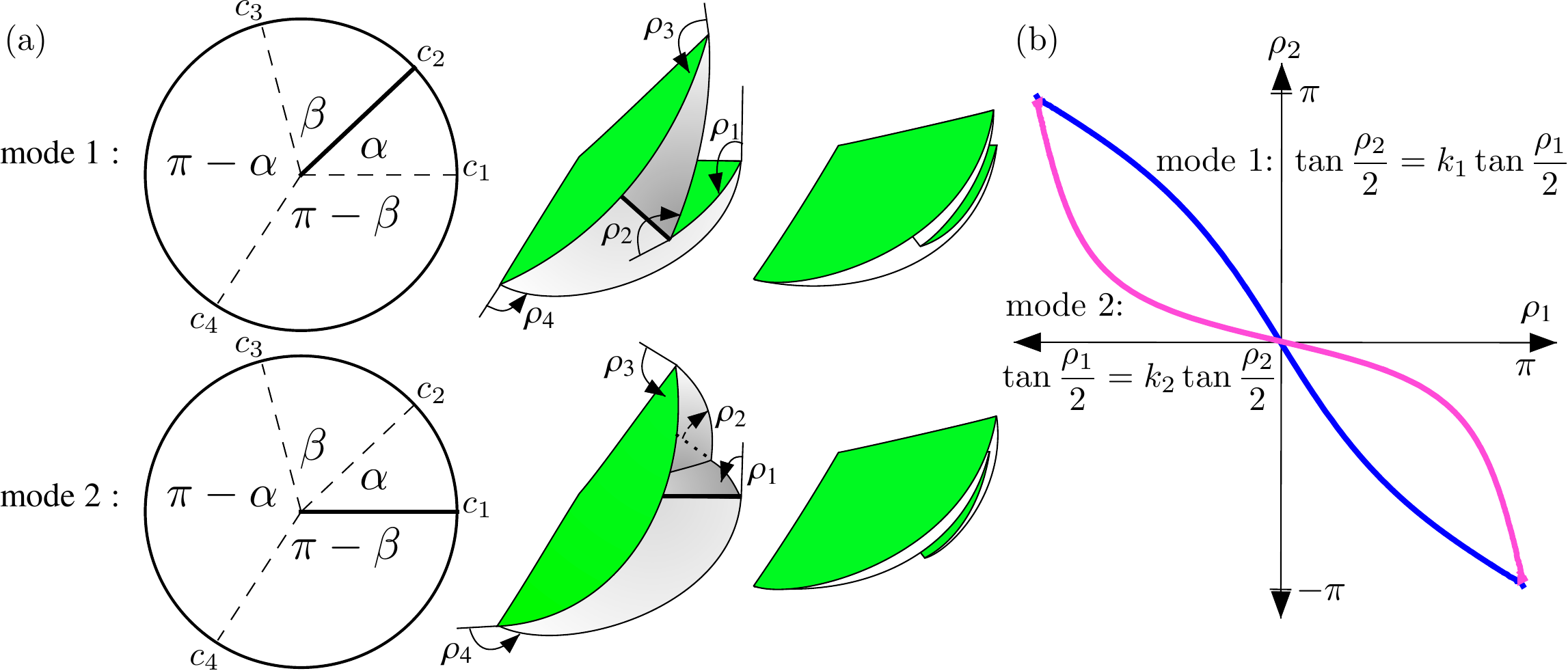}}
\caption{(a) The two folding modes of flat-foldable, Euclidean degree-4 vertices, crease patterns and folding images. (b) Folding angle $(\rho_1,\rho_2)$ plots of their folding angle equations from Theorem~\ref{thm:FF}.}\label{fig2}
\end{figure}

\begin{theorem}\label{thm:FF}
Let $(\rho_1,\ldots,\rho_4)$  be the folding angles of a Euclidean, flat-foldable degree-4 rigid origami vertex $C=(\alpha,\beta,\pi-\alpha,\pi-\beta)$, where we assume that $\alpha<\pi/2$ (see Figure~\ref{fig2}). 
Then when rigidly folding $C$ the folding angles must follow one of two algebraic relationships, called the {\em folding modes} of $C$:
\begin{eqnarray}
\mbox{mode 1:} & \rho_1=\rho_3, \rho_2=-\rho_4, \mbox{ and }\tan(\rho_2/2)=-k_1 \tan(\rho_1/2)\\
\mbox{mode 2:} & \rho_2=\rho_4, \rho_1=-\rho_3, \mbox{ and }\tan(\rho_1/2)=k_2 \tan(\rho_2/2)
\end{eqnarray}
where $k_1, k_2$ are constants determined by $\alpha$ and $\beta$. Specifically, 
$$k_1(\alpha,\beta)=\frac{\cos((\alpha+\beta)/2)}{\cos((\alpha-\beta)/2)}=  \frac{1-\tan(\alpha/2)\tan(\beta/2)}{1+\tan(\alpha/2)\tan(\beta/2)} \mbox{ and }$$
$$k_2(\alpha,\beta)=\frac{\sin((\alpha-\beta)/2)}{\sin((\alpha+\beta)/2)}=\frac{\tan(\alpha/2)-\tan(\beta/2)}{\tan(\alpha/2)\tan(\beta/2)}.$$
\end{theorem}

Here we express the constants $k_1$ and $k_2$ in ways commonly found in the literature.
See Figure~\ref{fig2} for a visualization of Theroem~\ref{thm:FF}, and see \cite{Origametry,TachiHull} for proofs. The main take-away is that in the Euclidean, flat-foldable case of degree-4 vertices, the folding angle relationships are linear under the half-angle substitution $t_i=\tan(\rho_i/2)$. However, also notice that constants $k_1$ and $k_2$ possess a certain amount of symmetry, namely 
$k_1(\alpha,\beta) = k_2(\pi-\alpha,\beta)=-k_2(\alpha,\pi-\beta)$ and
$k_2(\alpha,\beta)=k_1(\pi-\alpha,\beta)=-k_1(\alpha,\pi-\beta)$.  This symmetry reflects how determining whether we are folding in mode 1 or 2 depends on our choice for $\alpha$.

In contrast, non-flat-foldable degree-4 rigid origami vertex crease patterns have more complicated kinematic equations, such as those seen in \cite{Ivan}. Because they will be used in this paper, we report the more recent equations found in \cite{FHK}:

\begin{theorem}\label{thm:FHK}
Let $(\rho_1,\ldots,\rho_4)$  be the folding angles of a general degree-4 rigid origami vertex $C=(\alpha_1,\ldots, \alpha_4)$. 
Then the relationships between the folding angles $\rho_i$ are described as follows: for $i=1,\ldots, 4$, with indices taken cyclicly under addition, with $\rho_i$ at the crease between $\alpha_i$ and $\alpha_{i+1}$ and setting  $t_i = \tan(\rho_i/2)$, we have

\begin{equation}\label{THARPeq1}
t_i^2 = 
\frac{-(1+ t_{i+2}^2 )\cos(\alpha_{i-1}+\alpha_i)+ t_{i+2}^2 \cos(\alpha_{i+1}-\alpha_{i+2})+\cos(\alpha_{i+1}+\alpha_{i+2})}
{(1+t_{i+2}^2 )\cos(\alpha_{i-1}-\alpha_i)-t_{i+2}^2 \cos(\alpha_{i+1}-\alpha_{i+2})-\cos(\alpha_{i+1}+\alpha_{i+2})}
\end{equation}
and
\begin{equation}\label{THARPeq2}
\begin{split}
\cos\alpha_{i+2} \left(1+t_{i}^2\right)\left(1+t_{i+1}^2\right) =\quad 
&\cos(\alpha_{i+1}-\alpha_i-\alpha_{i-1}) t_{i+1}^2 
+ \cos(\alpha_{i+1}+\alpha_i-\alpha_{i-1})t_{i+1}^2 \\
+ &\cos(\alpha_{i+1}-\alpha_i+\alpha_{i-1}) t_{i}^2 t_{i+1}^2
+ \cos(\alpha_{i+1}+\alpha_{i}+\alpha_{i-1}) \\
+ & 4\sin\alpha_{i+1} \sin\alpha_{i-1} t_{i} t_{i+1}.
\end{split}
\end{equation}
\end{theorem}

Equation~\eqref{THARPeq1} describes the relationship between opposite folding angles, while Equation~\eqref{THARPeq2} captures the relationship between adjacent folding angles.

\section{Results}\label{sec2}

We are now ready to state the main definition of this paper.

\begin{definition}\label{def1}
If $C=(\alpha_1,\alpha_2,\alpha_3,\alpha_4)$ is a degree-4 rigid origami vertex, then we call the vertex $C^*=(\pi-\alpha_1,\pi-\alpha_2,\pi-\alpha_3,\pi-\alpha_4)$ the \textit{dual vertex} of $C$.
\end{definition}

We make a few elementary observations:

First, if $C$ is an elliptic vertex, then $C^*$ is hyperbolic, and vice versa.

Second, if $C$ is a flat-foldable vertex, then so is $C^*$ (by Kawasaki's Theorem).
  If $C$ is also Euclidean, then $C=C^*$ (again by Kawasaki's Theorem). That is, flat-foldable Euclidean degree-4 vertices are \textit{self-dual}.
  
\begin{example}\label{ex1}

\begin{figure}
\centerline{\includegraphics[width=\linewidth]{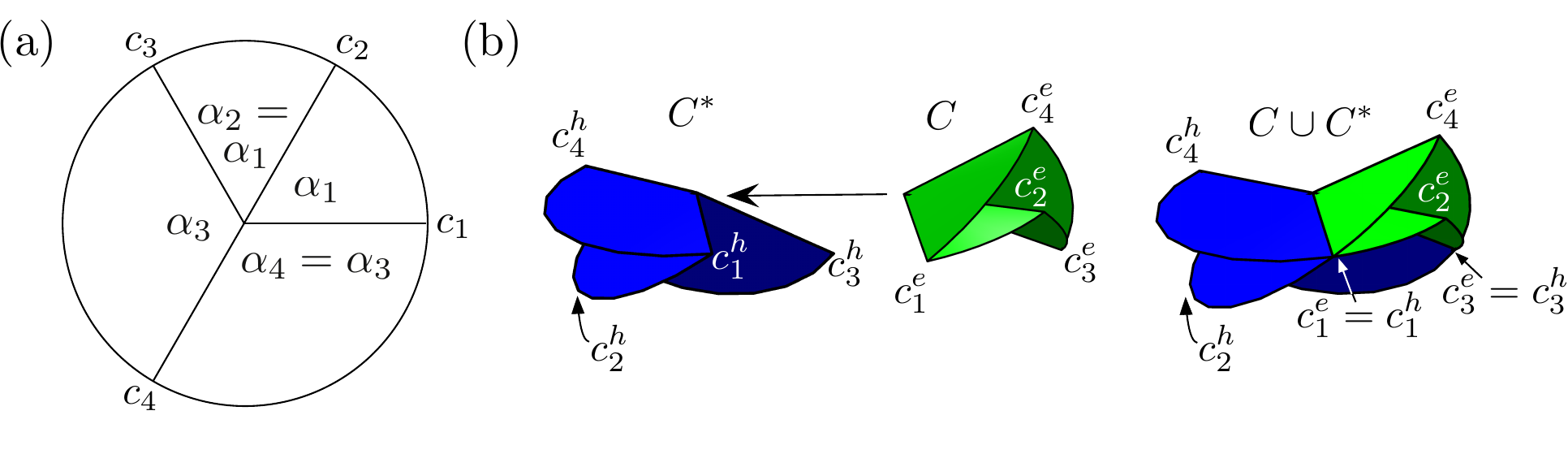}}
\caption{\em (a) A bird's foot vertex, where $\sum\alpha_i$ need not equal $2\pi$. (b) An elliptic $C$ being joined to its dual $C^*$ to show that $C\cup C^*$ is the intersection of two mountain-folded discs.}\label{fig3}
\end{figure}

We motivate Definition \ref{def1} and prepare for what follows by way of an example. Suppose that $C$ is a \textit{bird's foot} vertex, meaning that $\alpha_1=\alpha_2$ and $\alpha_3=\alpha_4$ (and thus the crease pattern looks somewhat like a bird's foot). An example is shown in Figure~\ref{fig3}(a).  Let $C$ be an elliptic bird's foot with creases $c_i^e$, and so $C^*$ will be a hyperbolic bird's foot, and we denote its creases by $c_i^h$.  We flex $C$ and $C^*$ so that the angle between the creases $c_1^e$ and $c_3^e$ is the same as the angle between $c_1^h$ and $c_3^h$.  Then we merge $C$ and $C^*$, identifying $c_1^e=c_1^h$ and $c_3^e=c_3^h$. This is shown in Figure~\ref{fig3}(b), where the images were created in Mathematica using folding angles derived from Equations~\eqref{THARPeq1} and \eqref{THARPeq2}.  This combined vertex, which we denote $C\cup C^*$, looks like two mountain-folded discs of paper intersecting, i.e., the merged creases $c_1^e=c_1^h$ have the appearance of intersecting planes, and thus have equal (up to sign) folding angles $\rho_1^e$ and $\rho_1^h$. The same seems to hold for $c_3^e=c_3^h$. 

However, we could have performed the operation in Figure~\ref{fig3}(b) in reverse; if we take two discs of paper that are mountain-folded with equal folding angles and intersect them along their creases at some angle, then we may separate this along the lines of intersection to create an elliptic degree-4 vertex $C$ and a hyperbolic degree-4 vertex $C^*$ that are duals of each other. By construction, the corresponding folding angles of these two degree-4 vertices will be equal (up to sign).
    
\end{example}    

Example \ref{ex1} generalizes to non-bird's foot vertices and can be proven independently of Mathematica pictures. We capture this in the following Theorem.
    

\begin{theorem}
\label{thmduality}
The rigid folding kinematics of a degree-4 vertex $C$ and its dual $C^*$ are essentially equivalent.  By this we mean that the configuration spaces $S(C)$ and $S(C^*)$ are the same up to a change of sign in the coordinates.
\end{theorem}

\begin{proof}
The gist of the proof is to show that equations representing the relationship between the folding angles of degree-4 rigid origami vertex are symmetric under the transformation $\alpha_i\mapsto\pi-\alpha_i$. However, like most kinematic equations one finds in mechanics, folding angle equations can be quite convoluted (such as those in \cite{Ivan}) and do not easily exhibit this symmetry.

The recent degree-4 rigid origami vertex equations from \cite{FHK}, shown in Equations~\eqref{THARPeq1} and \eqref{THARPeq2} above, readily prove this symmetry. 
Since $\cos(x\pm y) = \cos((\pi-x)\pm(\pi-y))$, we immediately have that Equation~\eqref{THARPeq1} is the same equation that describes the opposite folding angle relationships of the dual vertex $C^*=(\pi-\alpha_1,\pi-\alpha_2,\pi-\alpha_3,\pi-\alpha_4)$.  That is, if $\tau_i$ is the folding angle of the crease between the sector angles $\pi-\alpha_{i-1}$ and $\pi-\alpha_{i}$ in $C^*$, then we have $(\rho_1,\rho_3)=(\pm\tau_1,\pm\tau_3)$ and $(\rho_2,\rho_4)=(\pm\tau_2,\pm\tau_4)$, where the ``plus and minus" distinctions are due to the tangents being squared in Equation~\eqref{THARPeq1}.

On the other hand, $\cos(x\pm y\pm z)=-\cos((\pi-x)\pm(\pi-y)\pm(\pi-z))$ and $\cos x=-\cos(\pi-x)$ but $\sin x\sin y = \sin(\pi-x)\sin(\pi-y)$.  Therefore, Equation~\eqref{THARPeq2} describes the adjacent folding angle relationships of $C^*$ \textit{if and only if we have that} either $\tau_i=\rho_i$ and $\tau_{i+1}=-\rho_{i+1}$ or $\tau_i=-\rho_i$ and $\tau_{i+1}=\rho_{i+1}$.  This is consistent with the difference between a rigidly folded elliptic vertex $C$, which might have all its creases be mountains, and its hyperbolic dual $C^*$, which would have its creases alternate mountain-valley-mountain-valley.

In conclusion, we have that Equations~\eqref{THARPeq1} and \eqref{THARPeq2} give the same folding angle relationships for $C$ and $C^*$, and thus the same configuration spaces, up to a change of sign in the folding angles.  This completes the proof of Theorem~\ref{thmduality}.

\end{proof}

\begin{figure*}
\centerline{\includegraphics[width=\linewidth]{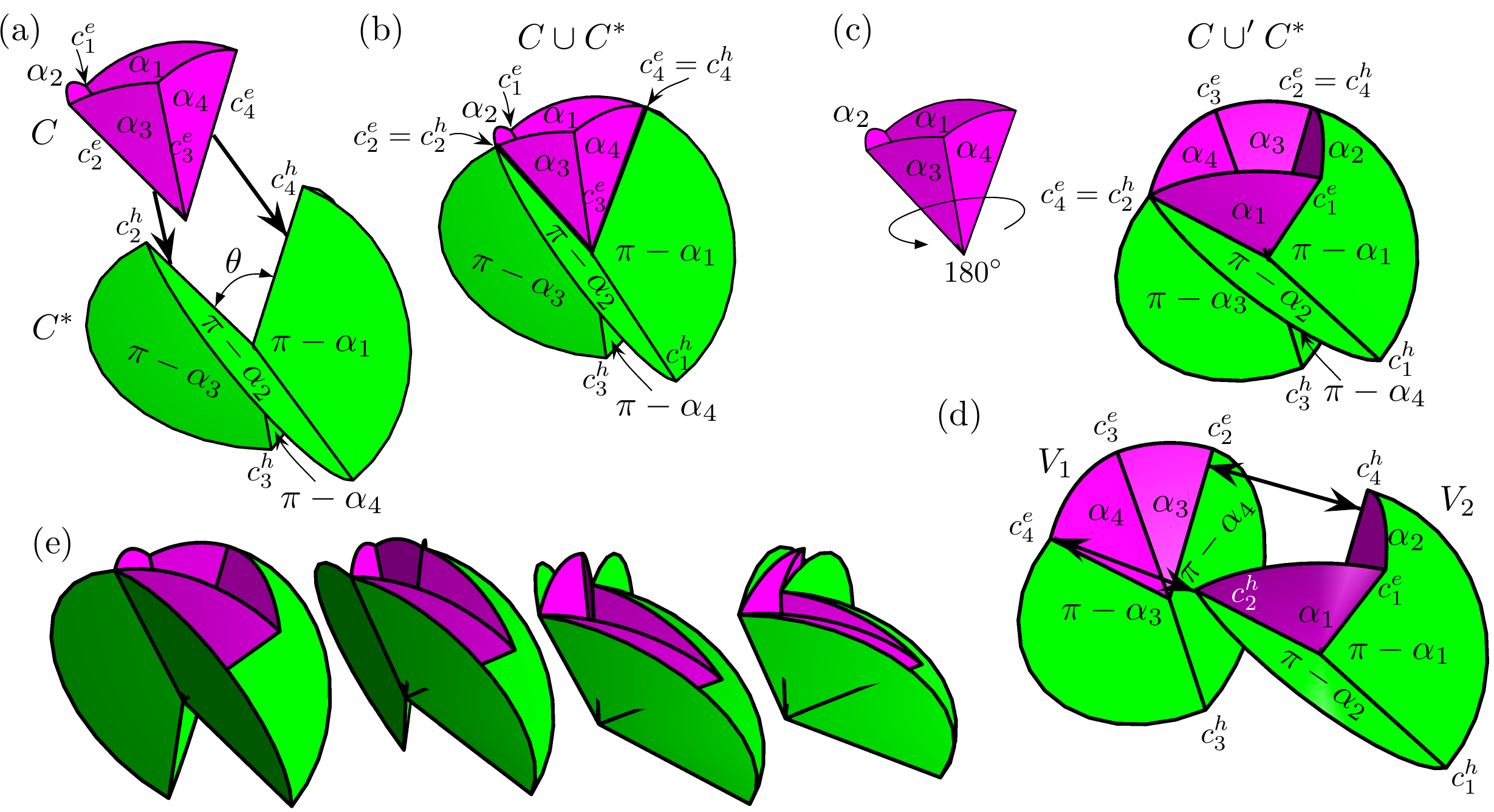}}
\caption{(a) An elliptic degree-4 vertex $C$ and its hyperbolic dual $C^*$. (b) The combined non-manifold vertex $C\cup C^*$. (c) Rotating $C$ $180^\circ$ to form the alternate combined vertex $C\cup' C^*$. (d) Splitting $C\cup' C^*$ into two flat-foldable vertices $V_1$ and $V_2$. (e) Illustrating the rigid folding of $C\cup' C^*$. See SI Movie S1 for animations.}\label{fig1-3}
\end{figure*}

\section{Geometric interpretation and flat-foldability}\label{sec3}

Similar to the bird's foot case in Example~\ref{ex1}, here is also a general geometric explanation of Theorem~\ref{thmduality}. Label the sector angles of an elliptic degree-4 vertex $C$ by $\alpha_i$ and its creases by $c_i^e$ as in Figure~\ref{fig1-3}(a), so that $c_i^e$ is between $\alpha_i$ and $\alpha_{i+1}$ (and $c_4^e$ is between $\alpha_4$ and $\alpha_1$). Then its dual $C^*$ will have corresponding sector angles $\pi-\alpha_i$, and let us denote its creases by $c_i^h$. If we synchronize the angle $\theta$ between $c_2^e, c_4^e$ in $C$ and $c_2^h, c_4^h$ in $C^*$ to be equal in their rigid foldings, then  we can identify the creases $c_2^e=c_2^h$ and $c_4^e=c_4^h$ to form a combined \textit{non-manifold vertex} $C\cup C^*$ as shown in Figure~\ref{fig1-3}(b). If we perform this ``gluing together" operation of $C$ and $C^*$ with the orientation as shown in Figure~\ref{fig1-3}(b), then the sector $\alpha_i$ in $C$ and $\pi-\alpha_i$ in $C^*$ will form a half-plane in $C\cup C^*$ (shown as a semicircle in Figure~\ref{fig1-3}(b), most easily seen with $\alpha_3$ and $\pi-\alpha_3$), and thus $C\cup C^*$ is actually the intersection of two folded planes (or two folded tacos in the figure). From this it is immediate that if $\rho_i^e$ and $\rho_i^h$ are the folding angles of the creases $c_i^e$ and $c_i^h$ respectively, then $\rho_i^e=\rho_i^h$ for $i=1, 3$  (since they are the same folded line in $C\cup C^*$) and $\rho_i^e=-\rho_i^h$ for $i=2,4$ (since those creases form the intersection of two planes). This is exactly the phenomenon we saw in Example~\ref{ex1}, except for general degree-4 vertices.

Note that when $C$, with a counterclockwise orientation of the $\alpha_i$ angles, is combined with $C^*$ as shown in Figure~\ref{fig1-3}(a) and (b), the full range of the rigid folding motion of $C\cup C^*$ might cause the material to self-intersect. (An example of this can be viewed in SI Movie S1.) However, reversing the orientation (i.e., placing the $\alpha_i$ angles in $C$ in clockwise order) will avoid such self-intersections (as shown in SI Movie S1) and still have the same kinematics since the synchronizing angle $\theta$ will be the same in either case. (This is true because in both cases we are identifying the creases $\{c_2^e,c_4^e\}$ with $\{c_2^h, c_4^h\}$.)

Now suppose we rotate the elliptic vertex $C$ by $180^\circ$ and attach it to $C^*$ so that $c_2^e=c_4^h$ and $c_4^e=c_2^h$, as in Figure~\ref{fig1-3}(c). We call this non-manifold vertex $C\cup' C^*$.  As in $C\cup C^*$, this alternate way of combining $C$ and $C^*$ will rigidly flex as we change the angle $\theta$ (as shown in Figure~\ref{fig1-3}(e)), but we may also decompose $C\cup'C^*$ into two flat-foldable vertices $V_1=(\alpha_1, \alpha_2,\pi-\alpha_1,\pi-\alpha_2)$ and $V_2=(\alpha_3,\alpha_4,\pi-\alpha_3,\pi-\alpha_4)$; see Figure~\ref{fig1-3}(d). Therefore \textit{the kinematics of an elliptic degree-4 vertex $C$ and its dual $C^*$ are related to the kinematics of the pair of flat-foldable vertices $(V_1,V_2)$}.

This is significant, and surprising. It is surprising because, why would anyone think that the folding of general degree-4 elliptic and hyperbolic polyhedral vertices are related to, or even controlled by, the Euclidean, flat-foldable special case? It is significant because as we have seen in Theorem~\ref{thm:FF}, the kinematics of degree-4 flat-foldable vertices can be algebraically expressed by a simple, linear relationship when re-parameterized by the tangent of half the folding angles. In fact, the Mathematica-generated images in Figure \ref{fig1-3}, below Figures \ref{fig5} and \ref{fig6}, and  the animations in supplemental movies S1 and S2 were all made using the more simple flat-foldable degree-4 folding angle equations from Theorem~\ref{thm:FF}.

\begin{figure*}
\centerline{\includegraphics[width=0.7\linewidth]{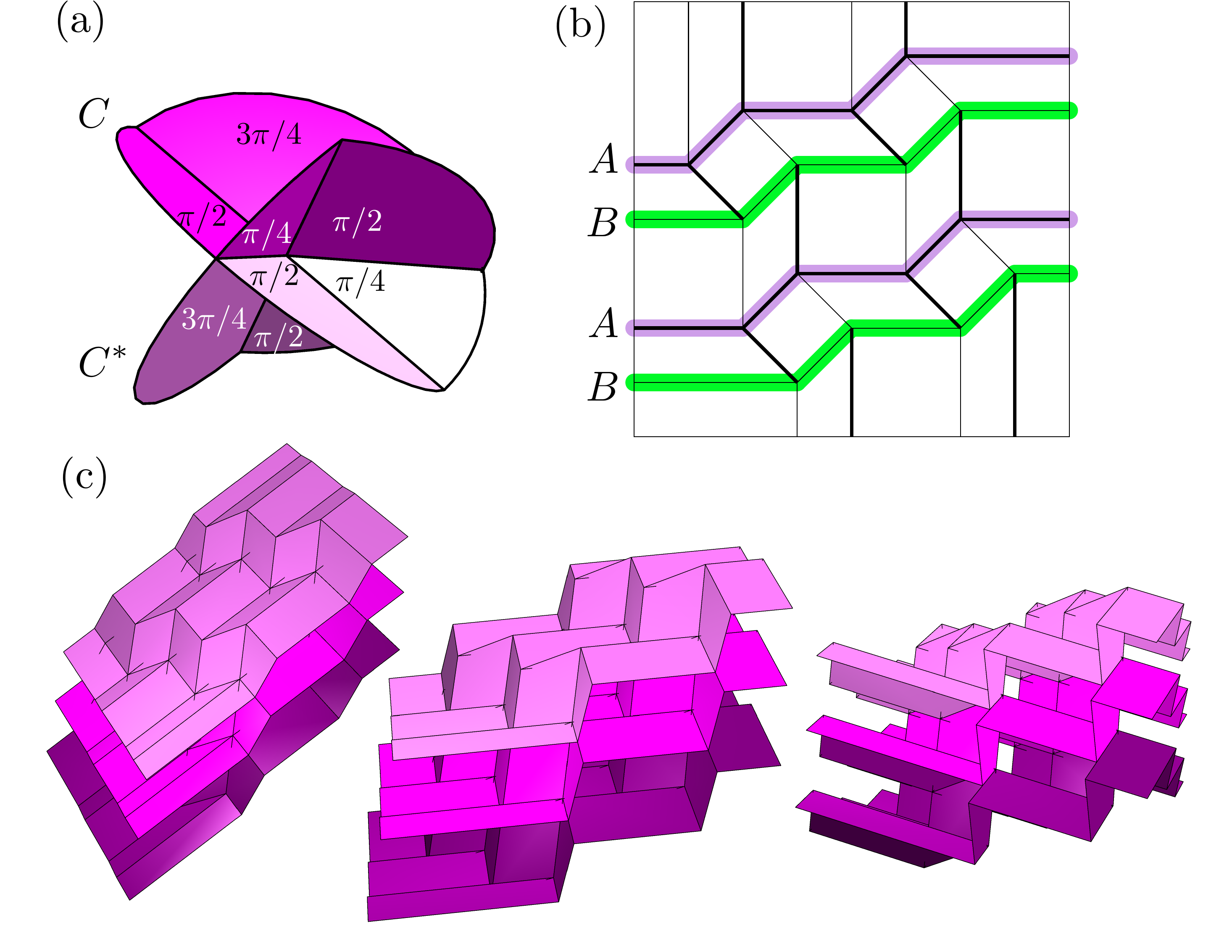}}
\caption{(a) The vertex $C=(\pi/4, \pi/2, 3\pi/4, \pi/2)$ with its dual $C^*$.  (b) The square twist origami tessellation with MV crease paths $A$ and $B$ highlighted. (c) Flexible CW structures made from the vertex in (a) by gluing mountain paths $A$ to valley paths $B$ of multiple sheets.}\label{fig5}
\end{figure*}

\section{CW complex generation}\label{sec4}

One potential application of the elliptic-hyperbolic duality is to use the combined vertices $C\cup C^*$ or $C\cup' C^*$ to create rigidly flexible 3D structures that form a continuously parameterized family of CW complexes in $\mathbb{R}^3$.  For example,  $C=(\pi/4, \pi/2, 3\pi/4, \pi/2)$ is the flat-foldable vertex used in the classic square twist crease pattern, which has been studied for its unusual mechanical and metamaterial features \cite{Silverberg2}. Combining $C$ with its dual gives the combined vertex $C\cup C^*$ (Figure~\ref{fig5}(a)). Tessellating this in 3-space is equivalent to gluing sheets of rigidly foldable square twist tessellations along corresponding rows of mountain and valley creases (Figure~\ref{fig5}(b)) and gives us the flexible structure shown in Figure~\ref{fig5}(c).  

If instead we let $C=(\pi/4,\pi/2,\pi/4,\pi/2)$ and create the $C\cup' C^*$ non-manifold vertex, then tessellating the combined vertex again gives us sheets of square twists, but this time they intersect to form a very different kind of flexible 3D structure (Figure~\ref{fig6}(a)-(b)).

\begin{figure*}
\centerline{\includegraphics[width=\linewidth]{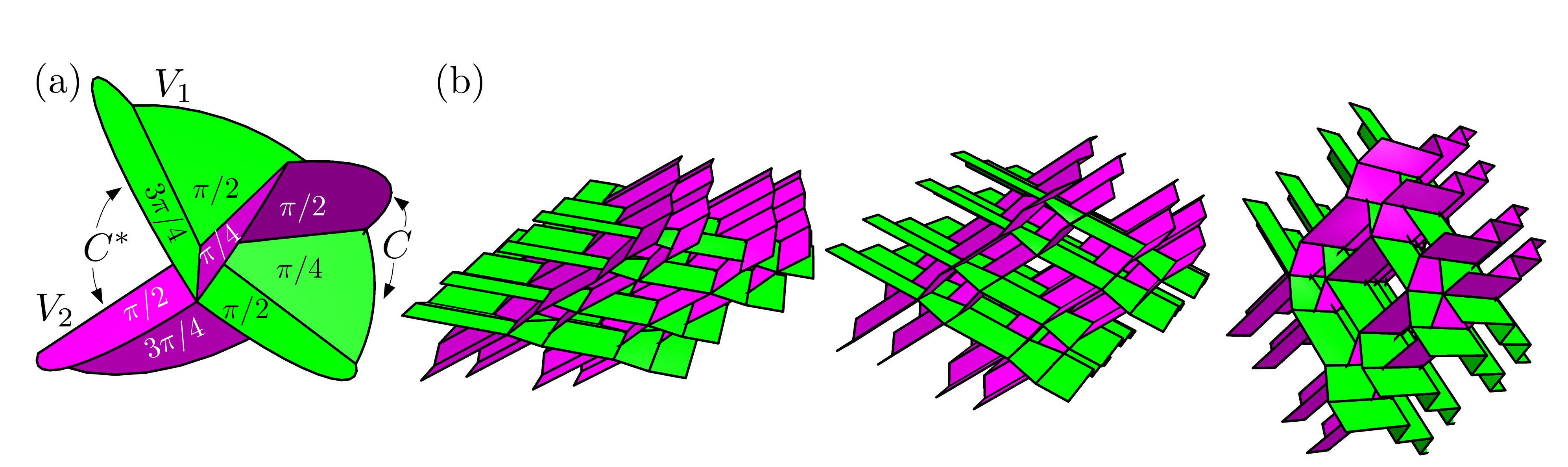}}
\caption{(a)  The vertex $C=(\pi/4, \pi/2, \pi/4, \pi/2)$ and its dual $C^*$. (b) Different flexes of the resulting flexible CW complex from (a). 
}\label{fig6}
\end{figure*}

\subsection{Auxetic behavior}\label{sec4.2}

The flexible CW complexes made by tessellating non-manifold $C\cup C^*$ or $C\cup' C^*$ vertices possess properties desired in metamaterials. For example, in the square twist $C\cup C^*$ vertex given by $C=(\alpha_1,\alpha_2,\alpha_3,\alpha_4)=(\pi/4, \pi/2, 3\pi/4, \pi/2)$ with its dual $C^*$ (Figure~\ref{fig5}), let $\rho$ be the folding angle of the ``major" creases of $C$, meaning the opposite pair of creases in $C$ that have the same MV parity and thus the same folding angle since here $C$ is flat-foldable and follows Theorem~\ref{thm:FF} (see \cite{Lang1} for more details on the ``major crease" terminology). In this example we have that $\rho$ is the folding angle between the $\alpha_1=\pi/4$ and $\alpha_2=\pi/2$ sector angles (which is the same as the crease between $\alpha_3=3\pi/4$ and $\alpha_4=\pi/2$). Then the tessellated flexible structure is contracting in two dimensions and expanding in the third for $0<\rho<\pi/2$ (since in this range the square twist sheets are rising from the flat, unfolded state into a 3D configuration) and contracting in all three dimensions for $\pi/2 < \rho < \pi$ (since the square twists will reach their maximum ``thickness" as 3D structures when the major folding angles around the central square are at right angles).  Thus, for this latter range of $\rho$ the flexible CW complex in Figure~\ref{fig5}(c) is contracting in all three dimensions and is therefore a three-dimensional auxetic structure with negative Poisson's ratio, a classic metamaterial property.

\section{Conclusion}\label{sec5}

The idea of combining rigid origami crease patterns to create flexible non-manifold 3D structures has been previously explored for engineering applications \cite{Tachi15,XingYou}. The duality of elliptic and hyperbolic degree-4 vertices presented here puts such prior work into a general, theoretical framework with many new directions for investigation. For example, the metamaterial properties of the 3D tessellations in Section~\ref{sec4} should be more fully explored. Furthermore, only two examples of such tessellated 3D  structures are given here; others are certainly possible. The field of rigidly-foldable origami tessellaitons that can be made from degree-4 vertices is quite rich \cite{Lang1,Lang2}. Therefore many other flexible 3D complexes can be explored from the dual vertex technique described here, giving a new tool for auxetic metameterial structure design.

In addition, the duality presented here points to a fundamental relationship among degree-4 rigid folding vertices and highlights the self-dual and influential role of the well-understood Euclidean flat-foldable case. This opens new questions: Can a moduli space of degree-4 rigid origami vertices with equivalent kinematics now be created? Can new folding angle equations for general degree-4 vertices be found via this connection to the flat-foldable case? Is a similar duality concept possible for higher-degree vertices?

%
%
%

\section*{Acknowledgements}

This work was supported in part by NSF grant DMS-2347000.

\bibliographystyle{plain}
\bibliography{Hull-EH-arXiv}

\end{document}